\documentclass[11pt,thmsa]{article}%
\usepackage{amsmath}
\usepackage{amsfonts}
\usepackage{amssymb}
\usepackage{graphicx}%
\newtheorem{theorem}{Theorem}[section]

\newtheorem{conjecture}[theorem]{Conjecture}

\newtheorem{example}[theorem]{Example}

\newtheorem{lemma}[theorem]{Lemma}

\newtheorem{proposition}[theorem]{Proposition}

\newenvironment{proof}[1][Proof]{\noindent\textbf{#1.} }{\ \rule{0.5em}{0.5em}}

\begin{document}
\date{}

\title{The  Landsberg equation of a Finsler space\thanks{Supported by NSFC (no. 11271216, 11271198, 11221091), Doctor fund of Tianjin Normal University (no. 52XB1305) and SRFDP of China}}
\author{Ming Xu$^1$ and Shaoqiang Deng$^2$\thanks{S. Deng is the corresponding author. E-mail: dengsq@nankai.edu.cn}\\
\\
$^1$College of Mathematical Sciences\\Tianjin Normal University \\Tianjin 300387\\People's  Republic of China\\
\\$^2$School of Mathematical Sciences and LPMC\\
Nankai University\\
Tianjin 300071\\
People's Republic of China}
\date{}
\maketitle

\begin{abstract}
Given a Finsler space, we introduce a system of partial differential equations, called the Landsberg equation.
Based on a careful analysis of the Landsberg equation and the observation that the solution space  is invariant under the linear isometries of the tangent Minkowski spaces,   we prove that an $(\alpha_1, \alpha_2)$-metric of the Landsberg type must be a Berwald metric. This shows that the  hunting  for a unicorn, one of the longest standing open problem in Finsler geometry, cannot be successful even in the very broad
class of $(\alpha_1,\alpha_2)$-metrics.


\end{abstract}

\section{Introduction}
The goal of this article is to consider a long standing open problem in Finsler geometry in the special class of $(\alpha_1,\alpha_2)$-metrics. In Finsler geometry,  there are two classes of special metrics which draw the attention of all
 the researchers in this field. Finsler spaces of the first type are called Landsberg spaces,
  which were studied by
G. Landsberg over one hundred years ago and were then named after him. Another type are called Berwald
spaces, which
were named after L. Berwald, who made significant progress in Finsler geometry
in the 1920's, in particular introduced the Berwald connection and two Berwald tensors.
 It is well known that every Berwald space
must be a Landsberg space. However, it has been one of the longest standing problem in Finsler geometry
whether there exists
a Landsberg space which is not a Berwald space. In 1996, Matsumoto (\cite{MA1})  found a list of rigidity results which
almost suggest that such metric does not exist. In 2003, Matsumoto emphasized this problem again and looked on it as
the most important open problem in Finsler geometry. Recently, D. Bao called such spaces  {\it unicorns} in Finsler
geometry (\cite{BA1}), mythical single-horned horse-like creatures which exist in legend but have never been seen by human beings.
There are a lot of unsuccessful attempts to find  explicit examples to solve this problem. The only
significant step
toward this problem is that some people constructed some examples of non-Berwaldian Landsberg metrics which are
either $y-local$ or not regular everywhere, i.e., metrics that  either are defined not on the whole tangent bundles or have some singularities (see \cite{AS}), which are not
 Finsler metrics in the strict sense. On the other hand, R. Bryant has announced in several occasions that in two-dimension, there
  is an abundance of such metrics, but in a bit more generalized sense, see \cite{BR} for the definition of generalized Finsler metrics. Also, Bao has advocated a perturbative approach to search for unicorns, see
   \cite{BA1}, \cite{BA2} for more information.

 In some special cases, this difficult problem is relatively simple and can be handled using some techniques. It is well-known that a Randers metric is of Landsberg type if and only if it is a Berwald metric (see, e.g., \cite{BCS}). In \cite{SH}, Z. Shen proved that a regular $(\alpha,\beta)$-metric on $M$ with $\dim M>2$ is Landsberg if and only it is
 Berwald. An $(\alpha,\beta)$-metric can be expressed as $F=\alpha \phi(\frac{\beta}{\alpha})$, where $\alpha$ is a
 Riemannian metric, $\beta$ is a one-form and $\phi$ is a smooth real function on ${\mathbb R}^1$.
 Shen's proof relies on many complicated calculations and the method does not apply to the general cases.
  Shen also constructed some examples
  of  $(\alpha,\beta)$-metrics which are almost regular, and which are Landsberg but not Berwald.

 More recently, Szabo made an argument to prove that any regular Landsberg space must be of the Berwald type
(\cite{SZ2}).
However, Matveev soon pointed out in \cite{MA} that there is a gap in Szabo's argument. As pointed out
in Szabo's correction to \cite{SZ2}, his argument only applies to the so called dual Landsberg
spaces. Hence the {\it unicorn} problem remains
open in Finsler geometry. Taking into account of so many unsuccessful efforts of so many researchers,
 one can say that this problem is becoming more and more
puzzling.

In this paper, we show that a Landsberg metric of the $(\alpha_1, \alpha_2)$-type must be
a Berwald metric. The notion of $(\alpha_1, \alpha_2)$-metrics was introduced by the authors in the previous paper \cite{dx1}. Our proof relies on a careful study of a system of linear partial differential equations, which is called the Landsberg equation and can be defined for any Finsler space. Since $(\alpha_1,\alpha_2)$-metrics are a very broad class of Finsler metrics, we believe that this technique can be applied to more generalized classes and will eventually lead to a complete solution of the unicorn problem.

The main results of this paper are the following two theorems.

\begin{theorem}\label{main-thm-1} The S-curvature of any Landsberg $(\alpha_1,\alpha_2)$-metric
vanishes identically.
\end{theorem}

\begin{theorem}\label{main-thm-2} Any Landsberg $(\alpha_1,\alpha_2)$-space
is a Berwald space.
\end{theorem}

We remark here that logically Theorem \ref{main-thm-1} is a corollary of Theorem \ref{main-thm-2}, since any Berwald space must have vanising S-curvature (see \cite{CS}). However, our proof of the second theorem relies heavily on the techniques developed in the proof of the first one.
Furthermore, the calculation involved in the proof for
Theorem \ref{main-thm-2} is essentially much harder, because
Shen's theorem in [13] has been applied.
 Therefore we state the first result as a
separate theorem.

One interesting insight  of this paper is that at any point an $(\alpha_1,\alpha_2)$-metric possesses much symmetry, and this is the main point in the proof of the main results.
Based on this observation, we make the following conjecture:

\begin{conjecture}
A homogeneous Landsberg space must be a Berwald space.
\end{conjecture}

Note that in the symmetric case the conjecture is automatically true (see \cite{DH2}). We refer the reader to \cite{DE12} for  fundamental properties of homogeneous Finsler spaces.

In Section 2 we give the necessary preliminaries and some known results. Section 3 is devoted to defining the normal coordinates for $(\alpha_1, \alpha_2)$-metrics. In Section 4, we consider the S-curvature of $(\alpha_1, \alpha_2)$-metrics. In Section 5,  we develop the theory of Landsberg equation for a general Finsler space. Finally, in Sections 6 and 7, we complete the proof of the main results of this paper.

\section{Preliminaries}

\subsection{Minkowski norm and Finsler metric}
Let $\mathbf{V}$ be an $n$-dimensional real vector space, with a basis
$\{e_1,\ldots,e_n\}$ and linear coordinates $y=y^i e_i$. A Minkowski norm $F$
on $\mathbf{V}$ is a continuous real  function $F$ on $\mathbf{V}$
satisfying the following conditions.
\begin{enumerate}
\item $F(y)\geq 0$, $\forall y\in \mathbf{V}$ and the equality hold if and  only if $y=0$.
\item $F(\lambda y)=\lambda F(y)$, for any  $\lambda> 0$.
\item $F$ is smooth on the slit space $\mathbf{V}\backslash\{0\}$, and the Hessian matrix
$$
(g_{ij}(y))=\left(\frac{1}{2}[F^2(y)]_{y^i y^j}\right)
$$
is positive-definite for any $y\in \mathbf{V}\backslash\{0\}$.
\end{enumerate}

A Finsler metric on an $n$-dimensional smooth manifold $M$ is a real continuous function
$F$ on $TM$  such that $F$ is smooth on $TM\backslash 0$ and
 the restriction of $F$ to each tangent space is a Minkowski norm. Generally, we  call $(M,F)$
a Finsler space or a Finsler manifold.

Here are some well known examples.

A Finsler manifold $(M,F)$ is called Riemannian, if $F$ is a quadratic function of $y$-coordinates
in each tangent space. In this case  the Hessian matrices define a smooth global section of
$\mathrm{Sym}^2 (TM^*)$ which is usually referred to as the Riemannian metric.

A Randers metric is  a Finsler metric of the form $F=\alpha+\beta$,
where $\alpha$ is a Riemannian metric and $\beta$
a 1-form whose length with respect to $\alpha$ is everywhere less than $1$. Randers metrics are the most important class of
non-Riemannian metrics in the study of Finsler geometry and its applications.

Another special class of Finsler metrics are  $(\alpha,\beta)$-metrics, which can be viewed as the generalization of Randers metrics. An $(\alpha,\beta)$-metric is a Finsler metric of the form
 $F=\alpha\phi(\beta/\alpha)$, where $\phi$ is a positive smooth function defined on an interval $(-b, b)\subset \mathbb{R}$, $b>0$,  and $\alpha$ and $\beta$ are respectively a Riemannian metric and a one-form on $M$. There are a lot of research work about $(\alpha,\beta)$-metrics in recent years.
See \cite{CS} for an exposition.

Recently  we  introduced another generalization of Randers metrics and $(\alpha,\beta)$-metrics,
 called $(\alpha_1,\alpha_2)$-metrics (\cite{dx1}). Now we briefly recall the definition of $(\alpha_1,\alpha_2)$-metrics.
Let $M$ be an $n$-dimensional manifold endowed with a Riemannian metric $\alpha$, and
an $\alpha$-orthogonal decomposition of the tangent bundle $TM=\mathcal{V}_1\oplus\mathcal{V}_2$,
where
$\mathcal{V}|_1$ and $\mathcal{V}|_2$ are two   linear subbundles with dimensions $n_1, n_2$ respectively, and
$\alpha_i=\alpha|_{\mathcal{V}_i}$ $i=1,2$  are naturally viewed as  functions on $TM$.
An  $(\alpha_1,\alpha_2)$-metric on $M$ is a  Finsler metric $F$ which can be written as
$F=\sqrt{L(\alpha_1^2,\alpha_2^2)}$.
We usually say that $F$ is   an $(\alpha_1,\alpha_2)$-metric with a dimension decomposition $(n_1,n_2)$.
An $(\alpha_1,\alpha_2)$-metric can also be represented as
$F=\alpha\phi(\alpha_2/\alpha)=\alpha\psi(\alpha_1/\alpha)$, in which
$\phi(s)=\psi(\sqrt{1-s^2})$. The condition for an $(\alpha_1,\alpha_2)$-metric to be smooth
and strongly convex can be found in \cite{dx1}.


Notice that in \cite{dx1}, we generally require that  $n_1>1$ and $n_2>1$ in the definition of an $(\alpha_1,\alpha_2)$-metric, since  otherwise it is (locally) a reversible $(\alpha,\beta)$-metric.  However, in this paper we will consider the related problems without this assumption.

\subsection{S-curvature and mean Cartan tensor}

Let  $(M,F)$ be  an $n$-dimensional Finsler space and $\{x=(x^i)\in M,y=y^j\partial_{x^j}\in TM_x\}$ be a local standard coordinate system on an open subset of $TM$.  The Busemann-Hausdorff
volume form can be
represented as $dV_{BH}=\sigma(x)dx^1\cdots dx^n$, where
$$
\sigma(x)=\frac{\omega_n}{\mbox{Vol}\{(y^i)\in\mathbb{R}^n|F(x,y^i\partial_{x^i})<1\}},
$$
in which $\mbox{Vol}$ denotes the volume of a subset with respect to the standard Euclidian metric on
$\mathbb{R}^n$, and $\omega_n=\mbox{Vol}(B_n(1))$. It is easily seen that the Busemann-Hausdorff
form is globally defined and does not depend on the specific coordinate system. On the other hand,
although the coefficient function $\sigma(x)$ is only locally defined which
depends on the choice of local coordinates $x=(x^i)$, the distortion function
\begin{equation}\label{-6}
\tau(x,y)=\ln\frac{\sqrt{\det(g_{ij}(x,y))}}{\sigma(x)}
\end{equation}
on $TM\backslash 0$ is independent of the local coordinates and globally defined.

The S-curvature $S(x,y)$ on $TM\backslash 0$ is defined as the derivative of
$\tau(x,y)$ in the direction of the geodesic spray, which is also a globally defined
vector field on $TM\backslash 0$. On  a local coordinate system, the geodesic spray can be expressed as
$G=y^i\partial_{x^i}-2G^i\partial_{y^i}$, where
$$
G^i=\frac{1}{4}g^{il}({[F^2]}_{x^k y^l}y^k-{[F^2]}_{x^k}).
$$

The derivatives of $\tau(x,y)$ in the $y$-directions define another non-Riemannian
curvature called mean Cartan tensor. On the same local coordinate system  as above, it can
be written as
$$
I_y(u)=u^i\partial_{y^i}\ln\sqrt{\det(g_{pq}(y))},\quad  u=u^i\partial_{y^i}.
$$
The Deicke's Theorem asserts that a Finsler metric $F$ is Riemannian if and only if the mean Cartan
tensor vanishes identically \cite{Dei}.

\subsection{Landsberg metric and Berwald metric}

A Finsler metric $F$ on $M$ is called a Landsberg metric,
if on any local coordinate system  $x=(x^i)\in M$ and $y=y^j\partial_{x^j}\in TM$, the geodesic
spray coefficients $G^i$s satisfy the following equations:
$$
y^j g_{ij}[G^i]_{y^p y^q y^r}=0, \quad \forall p,q,r.
$$

The metric $F$ is called a Berwald metric,  if on any local
coordinates, the geodesic spray coefficients $G^i$s are quadratic functions of
$y$-coordinates in each tangent space, namely,
$$
[G^i]_{y^p y^q y^r}=0, \quad \forall p,q,r.
$$

It is obvious that a Berwald space must be a Landsberg space. The problem whether there exists a non-Berwalidan Landsberg space has been one of the longest standing and most difficult problems in Finsler geometry. The reason to study Landsberg spaces and Berwald spaces is that they possess very  good geometric
properties. For example, on a Berwald space
parallel transformations define linear isometries between the tangent spaces. Meanwhile,
on a Landsberg space parallel transformations define isometries among the Riemannian manifolds
$T_x M\backslash\{0\}$, $x\in M$,  endowed with the Riemannian metric defined by the Hessian matrix $(g_{ij})$ (see for example, \cite{CS}).

\section{Normal coordinates for $(\alpha_1,\alpha_2)$-metrics}

In this section we define normal coordinates for $(\alpha_1, \alpha_2)$-spaces and study the behavior of the relevant geometric quantities of an $(\alpha_1,\alpha_2)$-space on  normal coordinates. Note that on a general Finsler space there does not exist normal coordinates as on a Riemannian manifold (see \cite{BCS}). However, since $(\alpha_1, \alpha_2)$-metrics are closely related to Riemannian metrics, one can apply this useful tool to the more general case.

  Let $F=\sqrt{L(\alpha_1^2,\alpha_2^2)}$ be an $n$-dimensional non-Riemannian $(\alpha_1,\alpha_2)$-metric on a manifold
$M$. Then we have a Riemannian metric $\alpha$, an $\alpha$-decomposition bundle decomposition $TM=\mathcal{V}_1\oplus\mathcal{V}_2$,  where $\mathcal{V}_i$ is a $n_i$-dimensional vector
bundle, and $\alpha_i=\alpha|_{\mathcal{V}_i}$, $i=1$, $2$,
are naturally viewed as functions on $TM$. The $\alpha$-orthogonal projection from $TM$
to $\mathcal{V}_1$ and $\mathcal{V}_2$ are denoted as $\mathrm{pr}_1$ and $\mathrm{pr}_2$
respectively.

Given ${p}\in M$, we can find a local coordinate system  $x=(x^i)$  and correspondingly
$y=y^j\partial_{x^j}\in TM$ on an open subset around $p$ satisfying the following conditions:
\begin{enumerate}
\item All the Christopher symbols of the Levi-Civita connection of $\alpha$ vanish at
$p=(0,\ldots,0)$.
\item The subspace $({\mathcal{V}_1})_{p}$ is linearly spanned by $\partial_{x^i}$,
$i=1$, $\ldots$, $n_1$, and $({\mathcal{V}_2})_{p}$ is linearly spanned  by $\partial_{x^i}$,
$i=n_1+1$, $\ldots$, $n$.
\item Restricted to $T_{p}M$, the functions $\alpha, \alpha_1, \alpha_2$ can be written as $\alpha^2=\mathop{\sum}\limits_{i=1}^n (y^i)^2$,
$\alpha_1^2=\mathop{\sum}\limits_{i=1}^{n_1}(y^i)^2$ and $\alpha_2^2=\mathop{\sum}\limits_{i=n_1+1}^n (y^i)^2$.
\end{enumerate}

A local coordinate system  satisfying the above conditions will be called  a normal chart at $p$ for the $(\alpha_1,\alpha_2)$-metric $F$.

Now fix a normal coordinate chart at $p$. Denote $\alpha_1^2(x)=a_{ij}(x)y^i y^j$ and $\alpha_2^2(x)=b_{ij}(x)y^i y^j$. Then there exist smooth coefficient functions $f_i^j(x)$
and $\tilde{f}_j^i(x)$, $i=1$, $\ldots$, $n_1$, $j=n_1+1$, $\ldots$, $n$, defined on an open subset containing  $p$,
 such that $({\mathcal{V}_1})_x$ is linearly spanned by
$\{\partial_{x^i}+f_i^j(x)\partial_{x^j}| 1\leq i\leq n_1\}$, and $({\mathcal{V}_2})_x$ is
linearly spanned by $\{\partial_{x^j}+\tilde{f}^i_j(x)\partial_{x^i}| n_1+1\leq j\leq n\}$
for any $x$ near $p$.
Since all these coefficient functions vanish at $p$, and
$\alpha^2(x)=(a_{ij}(x)+b_{ij}(x))y^i y^j=(\delta_{ij}+o(|x|))y^i y^j$,
we have
\begin{eqnarray*}
\mathrm{pr}_1\partial_{x^i} &=& \partial_{x^i} +\sum_{l=n_1+1}^n f_i^l
\partial_{x^l}+o(|x|),\\
\mathrm{pr}_1\partial_{x^j} &=& -\sum_{l=1}^{n_1}
\tilde{f}^l_j\partial_{x^l}+o(|x|),\\
\mathrm{pr}_2\partial_{x^i} &=& -\sum_{l=n_1+1}^n f^l_i\partial_{x^l}+o(|x|),\\
\mathrm{pr}_2\partial_{x^j} &=& \partial_{x^j}+\sum_{l=1}^{n_1}\tilde{f}^l_j\partial_{x^l}
+o(|x|),
\end{eqnarray*}
for $i\leq n_1$ and $j>n_1$.
On the other hand, since $\mathcal{V}_1$ and $\mathcal{V}_2$ are orthogonal with each other at any point,
we  get immediately that
$$
f^j_i(x) =-\tilde{f}^i_j(x)+o(|x|).
$$
A direct calculation then shows that
\begin{eqnarray*}
a_{ij}(x)&=& \langle \mathrm{pr}_1 \partial_{x^i},\mathrm{pr}_1 \partial_{x^j}\rangle_x
=a_{ij}(p)+o(|x|),\\
b_{ij}(x)&=& \langle \mathrm{pr}_2 \partial_{x^i},\mathrm{pr}_2 \partial_{x^j}\rangle_x
=b_{ij}(p)+o(|x|),
\end{eqnarray*}
for $i\leq j\leq n_1$ or $n_1<i\leq j$,
and that
\begin{eqnarray*}
b_{ij}(x) = \langle \mathrm{pr}_2 \partial_{x^i},\mathrm{pr}_2 \partial_{x^j}\rangle_x
= - f_i^j + o(|x|),
\end{eqnarray*}
for $i\leq n_1<j$.

The calculation above can be summarized as the following lemma.

\begin{lemma} Keep all notations above for the normal chart at $p$ for the
$(\alpha_1,\alpha_2)$-norm. Then we have
$$
[\partial_{x^k}a_{ij}](p)= -[\partial_{x^k}b_{ij}](p),$$
$$[\partial_{x^k}b_{ij}](p)= 0,\, \mbox{ for }i\leq j\leq n_1
\mbox{ or }n_1<i\leq j,$$
$$
[\partial_{x^k}b_{ij}](p)= -[\partial_{x^k}f_i^j](p)
=[\partial_{x^k}\tilde{f}_j^i], \,\mbox{ for } i\leq n_1<j.$$
\end{lemma}

On a normal chart at a point $p$, the constants
$[\partial_{x^k}b_{ij}](p)$ for $i\leq n_1<j$ ( or $j\leq n_1<i$ ) contain
the information of $\mathcal{V}_1$ and $\mathcal{V}_2$ which is crucial in the
 following study. For example, we have the following criterion for Berwald $(\alpha_1,\alpha_2)$-metrics on normal charts.

\begin{proposition} \label{proposition-3-2}
Let $F=\sqrt{L(\alpha_1^2,\alpha_2^2)}$
be an $(\alpha_1,\alpha_2)$-metric and
keep all the notations as above. If at any point $p\in M$, there exists a normal chart such
that $[\partial_{x^k}b_{ij}](p)=0$ for $i\leq n_1$ and $j>n_1$, then $F$ is a
Berwald metric.
\end{proposition}

The proof is a direction calculation. Just note that the condition that $[\partial_{x^k}b_{ij}](p)=0$
for $i\leq n_1$ and $j>n_1$ implies that $[\partial_{x^k}b_{ij}](p)=0$, $\forall i,j,k$.
Therefore
on a normal chart at $p$, we have
$$
[F^2]_{x^i}(p,y)=(L_2(\alpha_1^2,\alpha_2^2)-L_1(\alpha_1^2,\alpha_2^2))
[\partial_{x^i}b_{jk}](p)=0
$$
and
$$
G^i(p,y)=\frac{1}{4}g^{il}([F^2]_{x^k y^l}y^k-[F^2]_{x^l})(p,y)=0.
$$
Since a change of local coordinates will only cause a change of the spray coefficients $G^i$, $i=1,2,\cdots, n$, by quadratic functions
of the $y$-variables, we conclude that the geodesic spray coefficients are quadratic functions of the
coordinates. Therefore $F$ is a Berwald metric.

Locally, an $(\alpha_1,\alpha_2)$-metric satisfies the assumptions of
Proposition \ref{proposition-3-2} only when $\mathcal{V}_1$ and $\mathcal{V}_2$ are integrable and $\alpha$-parallel, that is, at any point $p$,  the integration submanifolds of $\mathcal{V}_1$ and $\mathcal{V}_2$ provide a local Riemannian product decomposition
$U=U_1\times U_2$ for the Riemannian metric $\alpha$, where $U$ is an open subset containing $p$. We will see later that, any non-Riemannian Landsberg $(\alpha_1,\alpha_2)$-metric (which is in fact a Berwald metric) must have such a local product decomposition.

The next  two examples  will be very useful for our  proof of the main theorems.

\begin{example}\label{e1}
We first consider  an $(\alpha_1,\alpha_2)$-metric
on $\mathbb{R}^2\backslash\{0\}$. Let $\alpha$ be the standard flat metric on $\mathbb{R}^2$, and $\mathcal{V}_1$ and $\mathcal{V}_2$ be the tangent sub-bundles
of $T(\mathbb{R}^2\backslash\{0\})$ spanned by $\partial_{\theta}$ and $\partial_r$
of the polar coordinates, respectively. Denote the restrictions of $\alpha$ to
$\mathcal{V}_i$ as $\alpha_i$, $i=1,2$. Then the function $F=\sqrt{L(\alpha_1^2,\alpha_2^2)}$ defines a $(\alpha_1,\alpha_2)$-metric
on $\mathbb{R}^2\backslash\{0\}$. It is Riemannian if and only if $L$ is a linear function. It is easy to check that,  on a normal chart at $p\in\mathbb{R}^2\backslash\{0\}$ for $F$,
we have $[\partial_{x^1}b_{12}](p)\neq 0$ and
$[\partial_{x^2}b_{12}](p)=0$.
\end{example}

\begin{example}\label{e2}
As another example, let us consider  an $(\alpha_1,\alpha_2)$-metric on $SU(2)\cong S^3$. Let $\alpha$ be
the standard metric of constant curvature on $S^3$ and $\mathcal{V}_2$  the
sub-bundle of $TS^3$ generated by the actions $Q\mapsto
\mathrm{diag}(\sqrt{-1}t,-\sqrt{-1}t)Q$ for $Q\in \mathrm{SU}(2)$. The integration curves of $\mathcal{V}_2$ are fibers of the Hopf map. Let $\mathcal{V}_1$ be the $\alpha$-orthogonal
complement of $\mathcal{V}_2$ and define the functions $\alpha_1$ and $\alpha_2$ accordingly.
 Then one can define an $(\alpha_1,\alpha_2)$-metric by $F=\sqrt{L(\alpha_1^2,\alpha_2^2)}$, where $L$ is a smooth function. If $L$ is not a linear function, then $F$ is non-Riemannian.
On a normal chart at  $p\in S^3$,  the corresponding function
$\alpha_2^2= b_{ij} y^i y^j$ satisfies the conditions
\begin{equation}\label{0004}
[\partial_{x^1}b_{13}](p)= [\partial_{x^2}b_{23}](p)=0,
\end{equation}
\begin{equation}\label{0005}
[\partial_{x^3}b_{13}](p)= [\partial_{x^3}b_{23}](p)=0,
\end{equation}
\begin{equation}\label{0006}
[\partial_{x^1}b_{23}](p)=-[\partial_{x^2}b_{13}](p)\neq 0.
\end{equation}

By the homogeneity of $F$, we only need to check the above identities at
$p=e$. A normal chart at $e$ for $F$ is given by
\begin{equation}
\exp: (x^1,x^2,x^3)\mapsto \exp\left(
\begin{array}{cc}
\sqrt{-1}x^3 & x^1+\sqrt{-1}x^2\\
        -x^1+\sqrt{-1}x^2 & -\sqrt{-1}x^3
\end{array}
\right).
\end{equation}
On an open subset containing $e$, the integration curves of $\mathcal{V}_2$ are given by
\begin{eqnarray}\label{0001}
& &\left(
\begin{array}{cc}
e^{\sqrt{-1}x^1} & 0\\
0 & e^{-\sqrt{-1}x^1}
\end{array}
\right)\cdot\exp
\left(
\begin{array}{cc}
\sqrt{-1}x^3 & x_1+\sqrt{-1}x_2\\
-x_1+\sqrt{-1}x_2 & -\sqrt{-1}x^3
\end{array}
\right)\nonumber\\
&=&\exp
\left(
\begin{array}{cc}
\sqrt{-1}(x^3 + t f_3) & (x_1+tf_1)+\sqrt{-1}(x_2+tf_2)\\
-(x_1+tf_1)+\sqrt{-1}(x_2+tf_2) & -\sqrt{-1}(x^3+tf_3)
\end{array}
\right).
\end{eqnarray}
Differentiating (\ref{0001}) with respect to $t$, and letting $t=0$, we can express the
unit vector field $\vec{v}=(f_1,f_2,f_3)|_{t=0}$
generating $\mathcal{V}_2$ in the normal chart at $e$ as
$$
\vec{v}(x^1,x^2,x^3)=(-x^2,x^3,1)+o(|x|),
$$
from which  the identities follow.
\end{example}

\section{S-curvature of $(\alpha_1,\alpha_2)$-metrics}

Let $F=\sqrt{L(\alpha_1^2,\alpha_2^2)}$ be an $(\alpha_1,\alpha_2)$-metric on $M$, with
a dimension decomposition $(n_1,n_2)$. In this section we will calculate the S-curvature
of $F$ at the given pair $(p,\mathbf{y})\in TM_{p}$. For simplicity, we
assume that $\mathbf{y}$ is not contained in $V_1\cap V_2$ and $\alpha(\mathbf{y})=1$.

We choose a normal chart for $F$ at $p$ with coordinates $x=(x^i)$, such that
$p=(0,\ldots,0)$ and $\mathbf{y}=(a,0\ldots,0,a')$.
For simplicity, we denote
$A_1=[\partial_{x^1}b_{1n}](p)$ and
$A_2=[\partial_{x^n}b_{1n}](p)$.
We keep all the notations for the derivatives of $L$ as in \cite{dx1}. For example, $L_1(\cdot,\cdot)$ and $L_{2}(\cdot,\cdot)$
are the two partial derivatives of $L$, $L_1$ and $L_2$ are their values at $(a^2,a'^2)$ respectively. Partial derivatives of $L$ with higher degrees are denoted similarly.

It is easy to see that,  at $(p,\mathbf{y})$, we have
\begin{eqnarray*}
\partial_{x^l}b_{ij}y^iy^jy^l&=&2a^2 a'A_1+2aa'^2 A_2,\\
\partial_{x^l}b_{1j}y^jy^l&=&a a'A_1+ a'^2 A_2,\\
\partial_{x^l}b_{nj}y^jy^l&=& a^2  A_1 + a a' A_2,\\
\partial_{x^1}b_{ij}y^iy^j&=& 2a a' A_1,\\
\partial_{x^n}b_{ij}y^iy^j&=& 2a a' A_2.
\end{eqnarray*}
%

Moreover, It is easy to calculate the following derivatives:
\begin{eqnarray*}
{[F^2]}_{y^i} &=& 2 y^i L_1 (\alpha_1^2,\alpha_2^2), \mbox{ when }i\leq n_1,\\
{[F^2]}_{y^i} &=& 2 y^i L_2 (\alpha_1^2,\alpha_2^2), \mbox{ when }i>n_1,\\
{[F^2]}_{y^i y^i} &=& 2 L_1 (\alpha_1^2,\alpha_2^2)+4 ({y^i})^2
L_{11}(\alpha_1^2,\alpha_2^2),\mbox{when } i\leq n_1,\\
{[F^2]}_{y^i y^i} &=& 2 L_2 (\alpha_1^2,\alpha_2^2)+4 ({y^i})^2
L_{22}(\alpha_1^2,\alpha_2^2),\mbox{when } i> n_1,\\
{[F^2]}_{y^i y^j} &=& 4 y^i y^j L_{11}(\alpha_1^2,\alpha_2^2),
\mbox{when } i<j\leq n_1\\
{[F^2]}_{y^i y^j} &=& 4 y^i y^j L_{22}(\alpha_1^2,\alpha_2^2),
\mbox{when } i>j> n_1,\\
{[F^2]}_{y^i y^j} &=& 4 y^i y^j L_{12}(\alpha_1^2,\alpha_2^2),
\mbox{when } i\leq n_1<j,\\
\end{eqnarray*}

Furthermore,  the coefficients of the fundamental tensor at $\mathbf{y}$ are given by
\begin{eqnarray*}
g_{11} &=& L_1+2a^2 L_{11},\\
g_{nn} &=& L_2+2a'^2 L_{22},\\
g_{1n} &=& 2a a' L_{12},\\
g_{ii} &=& L_1,\,\, \forall i=1,\ldots,n_1,\\
g_{ii} &=& L_2,\,\, \forall i=n_1+1,\ldots,n,
\end{eqnarray*}
with all other $g_{ij}=0$. The inverse matrix of the Hessian matrix can be determined by
\begin{eqnarray*}
g^{11}&=& \frac{L_2+2a'^2 L_{22}}{L_1 L_2-2LL_{12}},\\
g^{nn}&=& \frac{L_1+2a^2 L_{11}}{L_1 L_2 -2LL_{12}},\\
g^{1n}&=& \frac{-2aa'L_{12}}{L_1 L_2-2LL_{12}},\\
g^{ii}&=& L_1^{-1},\,\, \forall i=2,\ldots,n_1,\\
g^{ii}&=& L_2^{-1}, \,\,\forall i=n_1+1,\ldots,n,
\end{eqnarray*}
with all other $g^{ij}=0$.

To calculate the coefficients of the mean Cartan torsion, we need to determine the
 coefficients of the Cartan tensor. A direct calculation shows that
\begin{eqnarray*}
C_{111} &=& 3a L_{11}+ 2a^3 L_{111},\\
C_{nn1} &=& a L_{12} + 2a a'^2 L_{221},\\
C_{n11} &=& a' L_{12} +2 a^2 a'L_{112},\\
C_{nnn} &=& 3 a' L_{22} + 2a'^3 L_{222},\\
C_{ii1} &=& a L_{11},\,\, \forall i=2,\ldots,n_1,\\
C_{iin} &=& a' L_{12},\,\, \forall i=2,\ldots,n_1,\\
C_{ii1} &=& a L_{12},\,\, \forall i=n_1+1,\ldots,n,\\
C_{iin} &=& a'L_{22}, \,\,\forall i=n_1+1,\ldots,n,
\end{eqnarray*}
with all other $C_{ijk}$ vanishing at $\mathbf{y}$.
 It follows that the coefficients $I_k=[\ln\sqrt{\det(g_{pq})}]_{y^k}$ of the mean Cartan torsion vanish at $\mathbf{y}$ except
\begin{eqnarray*}
I_1 &=&\frac{-LL_{12}-2a^2 LL_{112}}{a(L_1 L_2 -2 LL_{12})}
+(n_1-1)\frac{a L_{11}}{L_1}+(n_2-1)\frac{a L_{12}}{L_2},\\
I_n &=& \frac{-LL_{12}-2a'^2 LL_{122}}{a'(L_1 L_2-2LL_{12})}
+(n_1-1)\frac{a' L_{12}}{L_1}+(n_2-1)\frac{a' L_{22}}{L_2}.
\end{eqnarray*}
Moreover, using the facts that $aI_1+a'I_n=0$ and that
$$
(g^{il}(\mathbf{y})u_i)=\frac{1}{L_1 L_2-2LL_{12}}(a'L_2,0,\ldots,0,-aL_1),
$$
where $u=(u^i)=(a',0,\ldots,0,-a)$, one easily deduces that
\begin{eqnarray*}
I^l &=& 0, \,\,\mbox{for}\,\,l\neq 1,n,\\
I^1 &=& (\frac{-LL_{12}-2a^2 LL_{112}}{a(L_1 L_2 -2 LL_{12})}
+(n_1-1)\frac{a L_{11}}{L_1}+(n_2-1)\frac{a L_{12}}{L_2})\frac{L_2}{L_1L_2-2LL_{12}},\nonumber\\
\\
I^n &=& (\frac{-LL_{12}-2a'^2 LL_{122}}{a'(L_1 L_2-2LL_{12})}
+(n_1-1)\frac{a' L_{12}}{L_1}+(n_2-1)\frac{a' L_{22}}{L_2})\frac{L_1}{L_1L_2-LL_{12}}.
\nonumber\\
\end{eqnarray*}
\subsection{Calculation of $G_l$}
Denote $G_l=\frac{1}{4}({[F^2]}_{x^k y^l}y^k-{[F^2]}_{x^l})$. Then we have
\begin{eqnarray*}
{[F^2]}_{x^l}&=&L_1(\alpha_1^2,\alpha_2^2)(\alpha_1^2)_{x^l}
+L_2(\alpha_1^2,\alpha_2^2)(\alpha_2^2)_{x^l}\nonumber\\
&=& L_1(\alpha_1^2,\alpha_2^2)\partial_{x^l}a_{ij}y^i y^j
+L_2(\alpha_1^2,\alpha_2^2)\partial_{x^l}b_{ij}y^i y^j,
\end{eqnarray*}
and
\begin{eqnarray*}
{[F^2]}_{x^k y^l}y^k &=& (L_1(\alpha_1^2,\alpha_2^2)\partial_{x^k}a_{ij}y^i y^j
+L_2(\alpha_1^2,\alpha_2^2)\partial_{x^k}b_{ij}y^i y^j)_{y^l}y^k\nonumber\\
&=& \partial_{y^l}L_1(\alpha_1^2,\alpha_2^2)\partial_{x^k}a_{ij}y^iy^jy^k
+ \partial_{y^l}L_2(\alpha_1^2,\alpha_2^2)\partial_{x^k}b_{ij}y^iy^jy^k\nonumber\\
&+& 2 L_1(\alpha_1^2,\alpha_2^2)\partial_{x^k}a_{jl}y^jy^k
+ 2 L_2(\alpha_1^2,\alpha_2^2)\partial_{x^k}b_{jl}y^jy^k.
\end{eqnarray*}
Therefore we have
\begin{eqnarray*}
G_1(\mathbf{y})&=&aa'L_{12}A_1+\frac{1}{2}a'^2(L_2-L_1+2L_{12})A_2,\\
G_n(\mathbf{y})&=&\frac{1}{2}a^2(L_2-L_1-2L_{12})A_1-aa'L_{12}A_2.
\end{eqnarray*}

To summarize, we have
\begin{eqnarray}\label{6}
S(p,\mathbf{y})&=& [y^l\partial_{x^l}\tau ](p,\mathbf{y})-2 I^l G_l\nonumber\\
&=& [y^l\partial_{x^l}\tau ](p,\mathbf{y})-\Phi(\mathbf{y})(aL_1 A_1+a'L_2 A_2)+\Psi(\mathbf{y})(a A_1+a' A_2),
\end{eqnarray}
where
$$
\Phi(\mathbf{y})=\frac{L}{L_1L_2-2LL_{12}}
\left[\frac{-LL_{12}-2a^2 LL_{112}}{aa'(L_1 L_2 -2 LL_{12})}
+(n_1-1)\frac{a L_{11}}{a'L_1}+(n_2-1)\frac{a L_{12}}{a'L_2}\right]
$$
and
$$
\Psi(\mathbf{y})=\frac{-LL_{12}-2a^2 LL_{112}}{aa'(L_1 L_2 -2 LL_{12})}
+(n_1-1)\frac{a L_{11}}{a'L_1}+(n_2-1)\frac{a L_{12}}{a'L_2}.
$$

\subsection{Calculation of $\partial_{x^l}\tau$}
On a  normal chart $(x^i)$ at $p$ for $F$,
the infinitesimal change of the Riemannian metric $\alpha$ at $p$ vanishes in any direction, and
the infinitesimal changes of the indicatrix of $F$ around $p$ belong to $\mathrm{SO}(n)$ with
respect to $\alpha(p,\cdot)$.
The function
$$
\sigma(x)=\frac{\omega_n}{\mbox{Vol}\{(y^i)\in\mathbb{R}^n|F(y^i\partial_{x^i})\leq 1)\}}
$$
in the Busemann-Hausedorff volumn form $\sigma(x)dx^1\cdots dx^n$ has a critical point
at $p$, so,
$$
[y^l\partial_{x^l}\sigma(x)](p,\mathbf{y})=0.
$$

Next we calculate $[y^l\partial_{x^l}\ln\sqrt{\det(g_{pq})}](p,\mathbf{y})$.
A direct calculation shows that
\begin{eqnarray*}
{[F^2]}_{x^l}&=& L_1(\alpha_1^2,\alpha_2^2)\partial_{x^l}a_{ij}y^i y^j
+L_2(\alpha_1^2,\alpha_2^2)\partial_{x^l}b_{ij}y^i y^j,\\
{[F^2]}_{x^l y^i}&=& 2L_{11}(\alpha_1^2,\alpha_2^2)\partial_{x^l}a_{jk}y^iy^jy^k+
2L_{12}(\alpha_1^2,\alpha_2^2)\partial_{x^l}b_{jk}y^iy^jy^k\nonumber\\
&+& 2 L_1(\alpha_1^2,\alpha_2^2)\partial_{x^l}a_{ij}y^j
+ 2 L_2(\alpha_1^2,\alpha_2^2)\partial_{x^l}b_{ij}y^j,\, \mbox{if } i\leq n_1,\\
{[F^2]}_{x^l y^i}&=& 2L_{12}(\alpha_1^2,\alpha_2^2)\partial_{x^l}a_{jk}y^iy^jy^k+
2L_{22}(\alpha_1^2,\alpha_2^2)\partial_{x^l}b_{jk}y^iy^jy^k\nonumber\\
&+& 2 L_1(\alpha_1^2,\alpha_2^2)\partial_{x^l}a_{ij}y^j
+ 2 L_2(\alpha_1^2,\alpha_2^2)\partial_{x^l}b_{ij}y^j,\,\, \mbox{if } i> n_1.
\end{eqnarray*}
Take the derivatives with respect to $y^i$ again and evaluate at $(p,\mathbf{y})$.
Then from the fact $[\partial_{x^j}a_{ii}](p)=-[\partial_{x^j}b_{ii}](p)=0$,   we deduce that, at $(p,\mathbf{y})$,
\begin{eqnarray*}
y^l\partial_{x^l}g_{ii} &=& -\frac{2a}{a'}L_{11}(aA_1+a'A_2),
\mbox{ if } 1<i\leq n_1,\\
y^l\partial_{x^l}g_{ii} &=& -\frac{2a}{a'}L_{12}(aA_1+a'A_2),
\mbox{ if } n_1<i<n,\\
y^l\partial_{x^l}g_{11} &=& (4a a'L_{112}+(\frac{2a'^3}{a}+6aa')L_{12})(aA_1+a'A_2),\\
y^l\partial_{x^l}g_{nn} &=& (\frac{4a^3}{a'}L_{112}+(\frac{4a}{a'}-6aa'-\frac{2a^3}{a'})L_{12})(aA_1+a'A_2),\\
y^l\partial_{x^l}g_{1n} &=& (-4a^2 L_{112}-4a^2 L_{12}+L_2-L_1)(aA_1+a'A_2).
\end{eqnarray*}
Then a direct  calculation shows that
\begin{eqnarray*}
[y^l\partial_{x^i}\ln\sqrt{\det(g_{pq})}](p,\mathbf{y}) = \frac{1}{2}[y^l g^{ij}\partial_{x^l}g_{ij}](p,\mathbf{y})
= -\Psi(\mathbf{y})(aA_1+a'A_2),
\end{eqnarray*}
from which we get
$$
S(p,\mathbf{y})=-\Phi(\mathbf{y})(a L_1 A_1+a' L_2 A_2).
$$

Now we summarize the above calculations to get the formula of  S-curvature. We Keep all notations
of the derivatives of $L$ as above.
\begin{theorem}\label{theorem-4-1}
Let $F=\sqrt{L(\alpha_1^2,\alpha_2^2)}$ be an $(\alpha_1,\alpha_2)$-metric on $M$. Given a normal chart
for $F$ at $p\in M$, such that
$\mathbf{y}=(a,0,\ldots,0,a')\in T_{p}M$ with $a\neq 0$, $a'\neq 0$ and $a^2+a'^2=1$,
 the S-curvature $S(p,\mathbf{y})$ is given by
$$
S(p,\mathbf{y})=-\Phi(\mathbf{y})(a L_1 A_1 + a' L_2 A_2),
$$
where
\begin{eqnarray*}
A_1 &=& [\partial_{x^1}b_{1n}](p),\\
A_2 &=& [\partial_{x^n}b_{1n}](p),\\
\end{eqnarray*}
and
$$
\Phi(\mathbf{y}) = \frac{L}{L_1L_2-2LL_{12}}
\left[\frac{-LL_{12}-2a^2 LL_{112}}{aa'(L_1 L_2 -2 LL_{12})}
+(n_1-1)\frac{a L_{11}}{a'L_1}+(n_2-1)\frac{a L_{12}}{a'L_2}\right].
$$
\end{theorem}

\subsection{Non-Riemannian $(\alpha_1,\alpha_2)$-metrics with vanishing S-curvature}

The formula of S-curvature  in Theorem \ref{theorem-4-1} implies that for a non-Riemannian
$(\alpha_1,\alpha_2)$-metric, the condition for S-curvature to be vanishing identically
is only relevant to Riemannian metric $\alpha$ and the $\alpha$-orthogonal decomposition of $TM$, rather than the function $L$ in the definition of the metric. This fact can be summarized
as the following theorem.

\begin{theorem}\label{thm-4-2}
Let $F$ be a non-Riemannian $(\alpha_1,\alpha_2)$-metric with vanishing S-curvature.
Then for any normal chart at any point $p$, we have
\begin{equation}\label{0020}
[\partial_{x^i}b_{jk}](p)+[\partial_{x^j}b_{ik}](p)=0,
\end{equation}
where $i\leq j\leq n_1<k$ or $k\leq n_1<i\leq j$.

Conversely, if for any point $p\in M$ and any normal coordinate chart at $p$,    (\ref{0020}) holds for $i=j$, then the S-curvature of $F$ is vanishing everywhere.
\end{theorem}
\begin{proof}
We first use the S-curvature formula to prove that $[\partial_{x^1}b_{1n}](p)=
[\partial_{x^n}b_{1n}](p)=0$, i.e., $A_1=A_2=0$. Assume conversely that
this is not true. Notice that a linear change of $x$-variable from $\mathrm{O}(n_1)\times \mathrm{O}(n_2)$
change a normal chart at $p$ for $F$ to another normal chart. Hence one can change the signs of $A_1$
and $A_2$ such that $A_1\geq 0$ and $A_2\geq 0$, by a suitable change of $x$-variables
of an orthogonal action which keeps all entries except that it may change the sign
of the first or the last one. If  $a>0$ and $a'>0$, then $L_1>0$ and $L_2>0$. So from the condition that S-curvature vanishes identically,
 it follows that $\Phi(y)$ is constantly $0$.  This implies that the mean Cartan tensor vanishes everywhere. Hence $F$ is Riemannian, which is a contradiction. Thus
$\partial_{x^1}b_{1n}=\partial_{x^n}b_{1n}=0$ at $p$. Changing the $x$-variables with
an orthogonal action which rotates the plane of the first two entries by $\pi/4$, and
keeping all other entries, we can then show that $\partial_{x^2}b_{1n}+\partial_{x^1}b_{2n}=0$
at $p$. By applying some similar arguments, one can prove (\ref{0020}) for other cases.
The last assertion follows directly from the formula of the S-curvature.
\end{proof}

\section{The Landsberg equation and local homogeneity}

In this section we will define the Landsberg equation of a Finsler space and deduce some important results which is useful for the proof of the main results of this paper.
\subsection{The Landsberg equation}

We first define the Landsberg equation for a Minkowski norm.

Let $F$ be a Minkowski norm on $\mathbb{R}^n$. With respect to
the standard basis $\{e_1,\ldots,e_n\}$ of $\mathbb{R}^n$, the Hessian matrix
$(g_{ij}(y))=\left(\frac{1}{2}[F^2]_{y^i y^j}\right)$ and its inverse are defined for the linear coordinates $y=y^i e_i$.

Given a vector valued function $\mathbf{f}=(f_1,\ldots,f_n):
\mathbb{R}^n\backslash\{0\}\rightarrow \mathbb{R}^n$, where the functions  $f_i$ are smooth
and positively homogeneous of degree 2, we define
$$
G^i(\mathbf{f})=\frac{1}{4}g^{il}(y^k\partial_{y^l}f_k -f_l).
$$
The Landsberg equation for the Minkowski norm $F$ is the following linear system of PDEs,
$$
y^j g_{ij}[G^i(\mathbf{f})]_{y^p y^q y^r}=0,\quad  p,q,r=1,2,\cdots,n.
$$

Let $\xi=(\xi_i^j)$ be an isomorphism of $\mathbb{R}^n$ with $\xi e_i=\xi ^j_i e_j$, and
$\eta=(\eta _i^j)$ the inverse of $\xi$ with $\eta e_i=\eta^j_i e_j$. Then we define the action of $\xi$
on $\mathbf{f}=(f_1,\ldots,f_n)$
as $(\xi\mathbf{f})(y)=(\xi^i_1 f_i(y'),\ldots,\xi^i_n f_i(y'))$,
where $y=y^i e_i$  and $y'=\xi y=\xi^i_j y^j e_i$.

The following lemma is the key observation of this paper.
\begin{lemma}
If the linear isomorphism $\xi$ preserves the Minkowski norm $F$, then the action of $\xi$ preserves
the solution space of the Landsberg equation of $F$.
\end{lemma}

\begin{proof}
First note that for $\xi=(\xi_i^j)$,
the linear functions $\tilde{y}^i=\xi^i_j y^j$ defines another linear coordinates on $\mathbb{R}^n$.
We denote the Hessian matrix of $F$ with respect to $\tilde{y}^i$  as $(\tilde{g}_{ij})$, and its inverse as $(\tilde{g}^{ij})$. Then we have
\begin{eqnarray*}
g_{ij} &=&\xi^k_i \xi^l_j\tilde{g}_{kl}, \quad \tilde{g}_{kl}=\eta^i_k \eta^j_l g_{ij},\\
g^{ij}&=&\eta^i_k \eta^j_l\tilde{g}^{kl}, \quad\tilde{g}^{kl}=\xi ^k_i \xi^l_j g^{ij}.
\end{eqnarray*}
If the Minkowski norm $F$ is preserved by $\xi$, i.e., if $F^2(y')=F^2(y)$, $\forall y\in \mathbb{R}^n$
and $y'=\xi y$, then we have  $F^2(\tilde{y}^1,\ldots,\tilde{y}^n)=F^2(y^1,\ldots,y^n)$ when we write $F$ as a function of $n$ variables.
Taking the second order partial derivatives  with respect to $\tilde{y}^i$ and $\tilde{y}^j$, we get
$$
g_{ij}(y')=\eta^k_i \eta^l_j g_{kl}(y)=\tilde{g}_{ij}(y),
$$
and $\tilde{g}^{ij}(y)=g^{ij}(y')$, $\forall i,j$.
Then for any vector valued function $\mathbf{f}=(f_1,\ldots,f_n)$, we have
\begin{eqnarray*}
G^i(\xi \mathbf{f})(y) &=& \frac{1}{4} g^{il}(y)(\partial_{y^l}(\xi ^m_k f_m(y'))y^k-\xi^h_l f_h(y'))
\nonumber\\
&=&\frac{1}{4} g^{il}(y) \xi^h_l (\tilde{y}^m\partial_{\tilde{y}^h}f_m(y')-f_h(y'))\nonumber\\
&=&\frac{1}{4} \eta^i_k\tilde{g}^{kh}(y)(\partial_{\tilde{y}^h}(f_m(y'))\tilde{y}^m-f_h(y'))
\nonumber\\
&=&\frac{1}{4} \eta^i_k g^{kh}(y')(\partial_{\tilde{y}^h}(f_m(y'))\tilde{y}^m-f_h(y'))\nonumber\\
&=& \eta^i_k G^k(\mathbf{f})(y').
\end{eqnarray*}
If $\mathbf{f}$ is a solution of the Landsberg equation of $F$, then we have
\begin{eqnarray*}
y^j g_{ij} [G^i(g\mathbf{f})]_{y^p y^q y^r} &=& \tilde{y}^j \eta^l_j g_{il}(y)
[\eta^i_k G(\mathbf{f})(y')]_{y^p y^q y^r}\nonumber\\
&=& \tilde{y}^j \tilde{g}_{jk}(y)[G^k(\mathbf{f})(y')]_{y^p y^q y^r}\nonumber\\
&=& \xi^{\bar{p}}_{p} \xi^{\bar{q}}_{q} \xi ^{\bar{r}}_r
[\tilde{y}^j g_{jk}(y')[G^k(\mathbf{f})]_{y^{\bar{p}}t^{\bar{q}}
y^{\bar{r}}}(y')]\nonumber\\
&=& 0.
\end{eqnarray*}
Therefore $\xi \mathbf{f}$ is also a solution of the Landsberg equation of $F$.
\end{proof}

Now we define the Landsberg equation of a Finsler space $(M,F)$. Let $p$ be any point on $M$ with a local
coordinates system $x=(x^i)$. Then in $T_{p}M$
endowed with the Minkowski norm $F(p,\cdot)$ and the $y$-coordinates $y=y^j\partial_{x^j}$, we can define the Landsberg equation as above, which will be called
the Landsberg equation of $F$ at $p$.
In the most important case that $\mathbf{f}(y)=\left([F^2]_{x^1}(p,y),\ldots,
[F^2]_{x^n}(p,y)\right)$, $G^i(\mathbf{f})$s are the geodesic spray coefficients,
and the Landsberg equation is the condition for $F$ to be a Landsberg metric at $p$.

We have thus proved the following

\begin{proposition}\label{proposition-4-2}
Let $(M,F)$ be a Finsler space, $p\in M$, and   $x=(x^i)$ a local coordinate system at $p$.
 Then we have
\begin{description}
\item{\rm (1)}\quad The Landsberg equation at $p$ is linear.
\item{\rm (2)}\quad The solution space of the Landsberg equation at $p$ is preserved under
the action of the group $L(T_{p}M,F(p,\cdot))$ of linear isometries of $(T_p(M), F(p, \dot))$.
\item{\rm (3)}\quad If $(M,F)$ is a Landsberg space, then
$\mathbf{f}(y)=([F^2]_{x^1}(p,y),\ldots,
[F^2]_{x^n}(p,y))$ is a solution of the Landsberg equation at any $p$.
\end{description}
\end{proposition}

\section{Proof of  Theorem \ref{main-thm-1}}\label{s6}

In this section, we will complete the proof of Theorem \ref{main-thm-1}. Assume that $F=\sqrt{L(\alpha_1^2,\alpha_2^2)}$ is a non-Riemannian
Landsberg $(\alpha_1,\alpha_2)$-metric on the $n$-dimensional manifold $M$ with dimension
decomposition $n=n_1+n_2$ and the decomposition of tangent bundle $TM=\mathcal{V}_1\oplus\mathcal{V}_2$.

Let
$\mathbf{y}\in T_{p}M\backslash ({\mathcal{V}_1}_{p}\cup
{\mathcal{V}_2}_{p})$. By the formula of  S-curvature, Theorem \ref{main-thm-1} follows  immediately if we can prove that $S(p,\mathbf{y})=0$.

It is easily seen that there exists a normal chart $x=(x^i)$ for $F$ at $p$ such that
the $y$-coordinates  of $\mathbf{y}$ are all $0$  except $y^1$ and $y^n$. It is not hard to see that, with respect
to the linear coordinates $y=y^j\partial_{x^j}$,  we have
$$
\mathrm{O}(n_1)\times \mathrm{O}(n_2)\subset L(T_{p}M,F(p,\cdot))\subset \mathrm{O}(n).
$$
We now prove the following lemma.

\begin{lemma}\label{lemma-6-1}
Keep all the notations as above. Then on  the plane defined by $y^2=\cdots=y^{n-1}=0$,
the vector valued function $\mathbf{f}=([F^2]_{x^1}(p,\cdot),0,\ldots,0)$ is a solution
of the following equation:
\begin{equation}\label{00001}
\sum_{i,j\in\{1,n\}}y^j g_{ij}[G^i({\mathbf{f}})]_{y^n y^n y^n}=0,
\end{equation}
where
\begin{eqnarray*}
G^1({\mathbf{f}}) &=&
\frac{1}{4} g^{11}( y^1 \partial_{y^1} {f}_1
- {f}_1) +
\frac{1}{4} g^{1n} y^1 \partial_{y^n} {f}_1, \nonumber\\
\\
G^n(\tilde{\mathbf{f}}) &=&
\frac{1}{4} g^{n1}( y^1 \partial_{y^1} {f}_1
- {f}_1) +
\frac{1}{4} g^{nn} y^1 \partial_{y^n} {f}_1. \nonumber\\
\end{eqnarray*}
\end{lemma}

\begin{proof}
By Proposition \ref{proposition-4-2}, $\mathbf{f}_0(y)=([F^2]_{x^1}(p,y),\ldots,[F^2]_{x^n}(p,y))$ is a solution of the Landsberg equation for $F$ at $p$.
Let $\xi $ be the orthogonal map which changes the signs of the all entries except the first
and the last ones. Then $\tilde{\mathbf{f}}=\frac{1}{2}(\mathbf{f}_0+\xi \mathbf{f}_0)$ is a solution of the Landsberg equation for $F$ with respect to the same normal chart at $p$. Restricted to the plane
 defined by $y^2=\cdots=y^{n-1}=0$, we have $\xi y=y$. Thus  $\tilde{\mathbf{f}}=(\tilde{f}_1,0,\ldots,0,\tilde{f}_n)$, where
$\tilde{f}_1(y)=[F^2]_{x^1}(p,y)$ and $\tilde{f}_n(y)=[F^2]_{x^n}(p,y)$
for $y=(y^1,0,\ldots,0,y^n)$. Now on the subspace defined by $y^2=\cdots=y^{n-1}=0$, the Landsberg equation implies that
\begin{equation}\label{0000}
\sum_{i,j\in\{1,n\}}y^j g_{ij}[G^i(\tilde{\mathbf{f}})]_{y^n y^n y^n}=0,
\end{equation}
in which
\begin{eqnarray*}
G^1(\tilde{\mathbf{f}}) &=&
\frac{1}{4} g^{11}( y^1 \partial_{y^1} \tilde{f}_1 + y^n \partial_{y^1} \tilde{f}_n
- \tilde{f}_1) +
\frac{1}{4} g^{1n}( y^1 \partial_{y^n} \tilde{f}_1 + y^n \partial_{y^n} \tilde{f}_n
- \tilde{f}_n), \nonumber\\
\\
G^n(\tilde{\mathbf{f}}) &=&
\frac{1}{4} g^{n1}( y^1 \partial_{y^1} \tilde{f}_1 + y^n \partial_{y^1} \tilde{f}_n
- \tilde{f}_1) +
\frac{1}{4} g^{nn}( y^1 \partial_{y^n} \tilde{f}_1 + y^n \partial_{y^n} \tilde{f}_n
- \tilde{f}_n). \nonumber\\
\end{eqnarray*}

The equation (\ref{00001}) is linear with respect to the function $\tilde{\mathbf{f}}(y)$ of
$y=(y^1,0,\ldots,0,y^n)$. On the other hand, given any vector $y=(y^1,0,\ldots,0,y^n)$,
it is easily seen that
$$
\tilde{f}_1 =2[\partial_{x^1} b_{1n}](p)y^1 y^n(L_2((y^1)^2,(y^n)^2)-L_1((y^1)^2,
(y^n)^2))
$$
and
$$
\tilde{f}_n = 2[\partial_{x^n} b_{1n}](p)y^1 y^n
(L_2((y^1)^2,(y^n)^2)-L_1((y^1)^2,(y^n)^2))
$$
are odd functions of $y^1$. Note that
$\tilde{\mathbf{f}}=\mathbf{f}+\mathbf{f}'$, where
$\mathbf{f}=(\tilde{f}_1,0,\ldots,0)$ and $\mathbf{f}'=(0,\ldots,0,\tilde{f}_n)$
are respectively the $y^1$-even and $y^1$-odd parts of the left side of (\ref{0000}).
Thus by the linearity, both $\mathbf{f}(y)$ and $\mathbf{f}'(y)$ satisfy  (\ref{0000})
at  $y=(y^1,0,\ldots,0,y^n)$.
In particular, the function $\mathbf{f}(y)=([F^2]_{x^1}(p,y),0,\ldots,0)$ satisfies the equation
(\ref{0000}) at $y=(y^1,0,\ldots,0,y^n)$.
\end{proof}

Now we consider the case that $[\partial_{x^1}b_{1n}](p)\neq 0$.
In this case,  we shall use the same non-linear function $L$
as in the proof of the above lemma to define an
$(\alpha_1,\alpha_2)$-metric on $\mathbb{R}^2\backslash\{0\}$. Let $\bar{\alpha}$ be
the standard flat Riemannian metric on $\mathbb{R}^2$, $\bar{\mathcal{V}}_1$ and $\bar{\mathcal{V}}_2$
be the subbundles of $T(\mathbb{R}^2\backslash\{0\})$ spanned  by $\partial_{\theta}$ and
$\partial_r$  with respect to  the polar coordinates, respectively, and ${\bar{\alpha}}_1$
and ${\bar{\alpha}}_2$ the restriction of $\bar{\alpha}$
on the subbundles $\bar{\mathcal{V}}_1$ and $\bar{\mathcal{V}}_2$, respectively. Then $\bar{F}=\sqrt{L({\bar{\alpha}}_1^2,{\bar{\alpha}}_2^2)}$ is
a non-Riemannian
$(\alpha_1,\alpha_2)$-metric on $\mathbb{R}^2\backslash\{0\}$.
It is easy to see that, on the normal chart
$\bar{x}=({\bar{x}}^i)$ and $\bar{y}={\bar{y}}^i \partial_{{\bar{x}}^i}$
 for $F'$ at a point $\bar{p}$, where ${\bar{\alpha}}_2 ^2(\bar{y})=
{\bar{b}}_{ij}{\bar{y}}^i {\bar{y}}^j$,
we have $[\partial_{{\bar{x}}^1}\bar{b}_{12}](\bar{p})\neq 0$ and
$[\partial_{{\bar{x}}^2}\bar{b}_{12}](\bar{p})=0$.
Moreover, on a normal chart  at any $\bar{p}$, the first factor ${\bar{f}}_1$
of  the vector
valued function
\begin{eqnarray}
\bar{\mathbf{f}}(\bar{y}) &=& ([\bar{F}^2]_{{\bar{x}}^1}(\bar{p},\bar{y}),[\bar{F}^2]_{\bar{x}^2}(\bar{p},\bar{y}))\nonumber\\
&=&({\bar{f}}_1,0),
\end{eqnarray}
is a scalar multiple of the first factor $f_1$ of $\mathbf{f}$ in Lemma \ref{lemma-6-1}, with $y=(a,0,\ldots,0,a')$ identified with $\bar{y}=(a,a')$. From Lemma \ref{lemma-6-1}, we see that $\bar{\mathbf{f}}$ must satisfy the
Landsberg equation for $\bar{F}$ at arbitrary $\bar{p}$. So $\bar{F}$ is also
a Landsberg metric. However,   we will show that this  is impossible in Section \ref{s8}.

Therefore  we  have $[\partial_{x^1}b_{1n}](p)=0$. Using a similar argument,
one can prove that $[\partial_{x^n}b_{1n}](p)=0$. By the formula of the S-curvature,
we have $S(p,\mathbf{y})=0$. This completes   the proof of
Theorem \ref{main-thm-1}.

\section{Proof of Theorem \ref{main-thm-2}}

In this section we shall complete the proof of Theorem \ref{main-thm-2}. As before, we only need to
deal with nonlinear $L$, i.e., we can assume that the Landsberg $(\alpha_1,\alpha_2)$-metric
$F=\sqrt{L(\alpha_1^2,\alpha_2^2)}$ is non-Riemannian everywhere. We will keep all the notations as in the last section.

Since the S-curvature of $F$ vanishes identically,  by Theorem
\ref{thm-4-2},
we have
$$
[\partial_{x^i}b_{jk}] (p)+[\partial_{x^j}b_{ik}](p) = 0,
$$
and
$$
[\partial_{x^{k}} b_{il}] (p)+[\partial_{x^l}b_{ik}](p)=0,
$$
for $1\leq i\leq j\leq n_1<k\leq l\leq n$.

To prove the  theorem, we need to show that for $i<j\leq n_1<k$, we have $[\partial_{x^i} b_{jk}](p)=0$, and for $j\leq n_1<k<l$, we have
$[\partial_{x^l} b_{jk}](p)=0$. These facts together with other properties of $b_{ij}$s at $p$ imply that for any $i, j,k$, we have $[\partial_{x^i}b_{jk}](p)=0$. Then by Proposition \ref{proposition-3-2}, we  conclude that the metric $F$ is
Berwaldian at any point where  it is non-Riemannian. Therefore  $F$ must be a Berwald metric, and Theorem \ref{main-thm-2} is proved.
If  $\dim M=2$, then Theorem \ref{main-thm-1} follows immediately from Theorem \ref{main-thm-2}. So
we can assume that $n=\dim M>2$, and without loss of generality, that $n_1\geq 2$.

We will only prove the identity $[\partial_{x^1}b_{2n}](p)=0$. The other cases can be treated using similar arguments. In the following we will view an $n\times n$ matrix as a linear transformation of $\mathbb{R}^n$ and vice versa.
Let  $\xi \in \mathrm{O}(n)$ be a matrix which change the signs of all entries except the first,
the second, and the last ones (e.g., if $n\geq 4$, $\xi=\mathrm{diag} (1, 1, -1,\cdots, -1, 1)$). Then the vector valued function
$\mathbf{f}=\frac{1}{2}(\mathbf{f}_0+\xi\mathbf{f}_0)$, where
$\mathbf{f}_0(y)=([F^2]_{x^1}(p,y),\ldots,[F^2]_{x^n}(p,y))$,  is a solution of the Landsberg equation of $F$ at $p$.
The restriction of $\mathbf{f}$ to $y=(y^1,y^2,0,\ldots,0,y^n)$ has the form
$\mathbf{f}=(f_1,f_2,0,\ldots,0,f_n)$, where
\begin{eqnarray*}
f_1(y) &=& [F^2]_{x^1}(p,y)\nonumber\\
&=& 2[\partial_{x^1}b_{2n}](p)(L_2((y^1)^2+(y^2)^2,(y^n)^2)-L_1((y^1)^2+(y^2)^2,(y^n)^2))y^2 y^n,\nonumber\\
f_2(y) &=& [F^2]_{x^2}(p,y)\nonumber \\
&=& 2[\partial_{x^2}b_{1n}](p)(L_2((y^1)^2+(y^2)^2,(y^n)^2)-L_1((y^1)^2+(y^2)^2,(y^n)^2))y^1 y^n,
\end{eqnarray*}
and
$f_n(y) = [F^2]_{x^n}(p,y) = 0$.
On the other hand, the restriction of the Hessian matrix $(g_{ij})$  to $y=(y^1,y^2,0,\ldots,0,y^n)$
  satisfies $g_{ij}\neq 0$, only when $i=j$, or both $i$ and $j$ lie in $\{1,2,n\}$. The similar assertions hold for the inverse matrix $(g^{ij})$.
Thus we have
\begin{equation}\label{0010}
\sum_{i,j\in\{1,2,n\}}y^j g_{ij} [G^j(\mathbf{f})]_{y^p y^q y^r},
\end{equation}
for all $p,q,r\in\{1,2,n\}$, where
\begin{eqnarray*}
G^1(\mathbf{f})&=&\frac{1}{4} g^{11}(y^1 \partial_{y^1}f_1 +y^2 \partial_{y^1}f_2
- f_1) \nonumber\\
&+&\frac{1}{4} g^{12}(y^1 \partial_{y^2} f_1 + y^2 \partial_{y^2} f_2 -f_2)
+\frac{1}{4} g^{1n}(y^1 \partial_{y^n} f_1 + y^2\partial_{y^n}f_2),\label{0011}\\
G^2(\mathbf{f})&=&\frac{1}{4} g^{21}(y^1 \partial_{y^1}f_1 + y^2\partial_{y^1}f_2 -f_1)
\nonumber\\
&+&\frac{1}{4} g^{22}(y^1 \partial_{y^2} f_1 + y^2 \partial_{y^2} f_2 -f_2)
+\frac{1}{4} g^{2n}(y^1 \partial_{y^n} f_1 + y^2\partial_{y^n}f_2),\label{0012}\\
G^n(\mathbf{f})&=&\frac{1}{4} g^{n1}(y^1 \partial_{y^1}f_1 + y^2\partial_{y^1}f_2 -f_1)
\nonumber\\
&+&\frac{1}{4} g^{n2}(y^1 \partial_{y^2} f_1 + y^2 \partial_{y^2} f_2 -f_2)
+\frac{1}{4} g^{nn}(y^1 \partial_{y^n} f_1 + y^2\partial_{y^n}f_2)\label{0013}.
\end{eqnarray*}

If $[\partial_{x^1}b_{2n}](p)=-[\partial_{x^2}b_{1n}](p)\neq 0$, then
we consider the following $(\alpha_1,\alpha_2)$-metric on $S^3$. Let $\alpha_0$ be the
standard Riemannian metric of constant curvature on $S^3$, $\bar{\mathcal{V}}_1$ and
$\bar{\mathcal{V}}_2$ be the subbundles defined in Example \ref{e2}, with corresponding $\bar{\alpha}_1$
and $\bar{\alpha}_2$. Then the same function $L$
defines a non-Riemannian $\bar{F}=\sqrt{L(\bar{\alpha}_1^2,\bar{\alpha}_2^2)}$ on $S^3$.
Fix a point $\bar{p}\in S^3$ and a normal chart at $\bar{p}$.
Note that $\bar{\alpha}_2^2=\bar{b}_{ij} \bar{y}^i \bar{y}^j$ satisfies
(\ref{0004})-(\ref{0006}). Consider   the vector valued function
\begin{eqnarray*}
\bar{\mathbf{f}}_0 (\bar{y}) &=& ([F^2]_{x^1}(\bar{p},\bar{y}),
[F^2]_{x^2}(\bar{p},\bar{y}),[F^2]_{x^3}(\bar{p},\bar{y}))\nonumber\\
&=& (\bar{f}_1,\bar{f}_2,0).
\end{eqnarray*}
Then the functions $\bar{f}_1$ and $\bar{f}_2$ differ from $f_1$ and $f_2$ given above by
the same scalar multiplication, if $y=(a_1,a_2,0,\ldots,0,a')$ is identified with
$\bar{y}=(a_1,a_2,a')$. Thus $\bar{\mathbf{f}}_0$ satisfies
the Landsberg equation for $\bar{F}$ at $\bar{p}$, which is the same
equations as (\ref{0010}). Therefore $\bar{F}$ is a Landsberg metric on $S^3$.
Notice that  $\bar{F}$ is a reversible $(\alpha,\beta)$-metric. By Proposition \ref{proposition-3-2}, $\bar{F}$ is a Berwald
metric on $S^3$, and $\bar{\mathcal{V}}_2$ is $\bar{\alpha}$-parallel with respect to the Chern connection of $\bar{F}$. This is impossible.
Therefore we have $[\partial_{x^1}b_{2n}](p)=0$. Similarly, we have
$[\partial_{x^i}b_{jk}](p)=0$, for all $i, j, k$.  Consequently
 $F$ is a Berwald metric. This completes   the proof of Theorem \ref{main-thm-2}.

\section{A special $(\alpha_1,\alpha_2)$-metric on
 $\mathbb{R}^2\backslash\{0\}$}\label{s8}
In this section, we will prove the assertion at the final part of Section 6. For this we consider a special $(\alpha_1,\alpha_2)$-metric on
$\mathbb{R}^2\backslash\{0\}$ defined as follows. Let
$\alpha(y)$ be the standard flat
Riemannian metric on $\mathbb{R}^2$. There are two $\alpha$-orthogonal subbundles
$\mathcal{V}_1$ and $\mathcal{V}_2$ of
$T(\mathbb{R}^2\backslash\{0\}$, spanned  by $\partial_{\theta}$ and $\partial_{r}$
with respect to the polar coordinates, respectively. Consider a non-Riemannian
$(\alpha_1,\alpha_2)$-metric $F=\sqrt{L(\alpha_1^2,\alpha_2^2)}$ on
$\mathbb{R}^2\backslash\{0\}$ with respect to the above decomposition.
For a normal chart for $F$ at a point $p\in\mathbb{R}^2\backslash\{0\}$, it is
easy to see that $[\partial_{x^1}b_{12}](p)\neq 0$ and
$[\partial_{x^2}b_{12}](p)=0$.

We now prove the following lemma.
\begin{lemma}
The $(\alpha_1,\alpha_2)$-metric defined above can not be a Landsberg metric.
\end{lemma}

\begin{proof}
Assume conversely that $F=\sqrt{L(\alpha_1^2,\alpha_2^2)}$ is a Landsberg metric.
Then each tangent space can be canonically identified with the Minkowski space $\mathbb{R}^2$
with the norm $F_0(y)=\sqrt{L((y^1)^2,(y^2)^2)}$ where the two $y$-coordinates correspond
to the subbundles $\mathcal{V}_1$ and $\mathcal{V}_2$, respectively. So we have a smooth
family of isometries $I_x:(T_x M, F(x,\cdot))\rightarrow (\mathbb{R}^2, F_0)$.
The points on the indicatrix of $(\mathbb{R}^2,F_0)$ can be parametrized as
$Y_t$ by the angle $t\in [0,2\pi)$. The anti-clockwise unit tangent vector at $U_t$, with respect to the Hessian matrix of $F_0$, will be denoted as $Y_t$.

Now we construct two subsets $\mathcal{S}_1$ and $\mathcal{S}_2$ of $[0,2\pi)$.

Let $\mathcal{S}_1$ be the set of  $t\in [0,2\pi)$ such that there exist
$x\in \mathbb{R}^2\backslash\{0\}$ and a unit tangent vector
$y\in T_x(\mathbb{R}^2\backslash\{0\})$
such that
$I_{x(s)} \dot{x}(s) \equiv Y_t$, where $x(s)$ is the unit-speed geodesic  $y$. Meanwhile, let $\mathcal{S}_2$ be the set of all $t\in [0,2\pi)$ such that there
exist $x\in\mathbb{R}^2\backslash\{0\}$ and a unit tangent vector $y\in T_x(\mathbb{R}^2\backslash\{0\})$ such that
$I_x y=Y_t$ and
$I_{x(s)} \dot{x}(s)$ is not constantly equal to $Y_t$, where $x(s)$ the unit speed geodesic  with initial vector  $y$. Obviously we have $\mathcal{S}_1\cup\mathcal{S}_2=[0,2\pi)$.

If we denote $I_{x(s)}=Y_{t(s)}$ for the unit-speed geodesic $x(s)$, then
$I_{x(s)}^{-1}U_{t(s)}$ is a linear parallel vector field along this geodesic.
By the assumption that $F$ is a Landsberg metric, $C_{t(s)}(U_{t(s)},U_{t(s)},U_{t(s)})$
is a constant function of $s$.
So $\mathcal{S}_2$ is a union of intervals, such that on each interval  $C_t(U_t,U_t,U_t)$
is locally a constant function.
By the relation between Cartan tensor and Landsberg tensor, we have
 $\frac{d}{dt}C_{Y_t}(U_t,U_t,U_t)=0$, $\forall t_0\in\mathcal{S}_2$.
On the other hand, by the S-curvature formula of Theorem \ref{theorem-4-1}, and the facts that $A_1\neq 0$ and $A_2=0$, we conclude that $t\in \mathcal{S}_1$ if and only if $\Phi(Y_t)=0$. Thus
$\mathcal{S}_1$ is a closed subset in $[0,2\pi)$.

If $\mathcal{S}_1=[0,2\pi)$, then the coefficient $\Phi(y)$ in
the S-curvature formula is constantly 0. This can happen only when the function $L$ defines
a Riemannian metric $F$, which is a contradiction. Thus $\mathcal{S}_1\ne [0,2\pi)$.
Now in any maximal
closed interval $\mathcal{U}$ contained in $\mathcal{S}_1$, the equation $\Phi=0$ gives a real analytic ODE satisfied by
$L$ (which can be changed to a function of the angle $t\in [0,2\pi)$ by its homogeneity).
So for $t\in\mathcal{U}$, $L$ is a real analytic function of $t$. Then
$C_{Y_t}(U_t,U_t,U_t)$ is also a real analytic function of $t$. But at least one end point of $\mathcal{U}$ is approached by intervals from $\mathcal{S}_2$, where $C_{Y_t}(U_t,U_t,U_t)$ is
a constant function. So at that end point, $C_{Y_t}(U_t,U_t,U_t)$ has a $0$-expansion, which
implies that
it is also a constant function on $\mathcal{U}$.

Since $\mathcal{S}_1\cup\mathcal{S}_2=[0,2\pi)$, and
$C_{Y_t}(U_t,U_t,U_t)$ is a locally constant function in both subsets, it is a constant
function of $t$. Now a direct computation shows that  at
any point with at least one $y^i$ equal to $0$, the Cartan tensor is equal to  $0$. Thus $C_{Y_t}(U_t,U_t,U_t)\equiv 0$ for all $t$.  Since the dimension of the Minkowski space is $2$, $F_0$ is a Riemannian norm and $F$ is a Riemannian metric,  which is a contradiction.
This completes the proof of the lemma
\end{proof}

Now we have finished the proof of the assertion at the end of Section \ref{s6}, concluding the proof of all the results of this paper.

\end{document}